\documentclass[a4paper,12 pt]{amsart}
\usepackage{hyperref}
\usepackage[usenames,dvipsnames,svgnames,table]{xcolor}
\usepackage[all]{xy}
 \usepackage{enumitem}
 \usepackage[normalem]{ulem}

\newcommand{\tsk}[1]{\textcolor{YellowOrange}}

\usepackage{tikz}
\usepackage{tikz-cd}

\makeatletter
\def\@endtheorem{\endtrivlist}
\makeatother

\usepackage[active]{srcltx}
\usepackage{amsmath}
\usepackage{mathtools}
\usepackage{amssymb}
\usepackage{amscd}
\usepackage{amsthm}
\usepackage[latin1]{inputenc}
\usepackage{mathrsfs}  

\usepackage{nicefrac}

\newtheorem{teo}{Theorem}[section]
\newtheorem{defin}[teo]{Definition}

\newtheorem{cor}[teo]{Corollary}
\newtheorem{conj}[teo]{Conjecture}
\newtheorem{lemma}[teo]{Lemma}
\theoremstyle{definition}

\newtheoremstyle{dico}
 {\baselineskip}   
  {\topsep}   
  {}  
  {0pt}       
  {} 
  {.}         
  {5pt plus 1pt minus 1pt} 
  {}          
\theoremstyle{dico}
\newtheorem{say}[teo]{}
\numberwithin{equation}{section}

\newcommand{\ra}{\rightarrow}
\newcommand{\C}{\mathbb{C}}
\newcommand{\R}{\mathbb{R}}
\newcommand{\Zeta}{{\mathbb{Z}}}

\newcommand{\QQ}{{\mathbb{Q}}}
\newcommand{\meno}{^{-1}}

\newcommand{\alfa}{\alpha}

\newcommand{\vacuo}{\emptyset}

\newcommand{\La}{\Lambda}

\newcommand{\restr}[1]          {\vert_{#1}}

\newcommand{\End}{\operatorname{End}}
\newcommand{\Ann}{\operatorname{Ann}}

\newcommand{\Lie}{\operatorname{Lie}}

\newcommand{\om}{\omega}

\renewcommand{\phi}{\varphi}
\newcommand{\lds}{\ldots}
\newcommand{\cds}{\cdots}
\newcommand{\cd}{\cdot}
\newcommand{\im}{\operatorname{im}}

\newcommand{\lra}{\longrightarrow}

\newcommand{\Ga}{\Gamma}

\newcommand{\lieg}{\mathfrak{g}}
\newcommand{\liem}{\mathfrak{m}}
\newcommand{\liep}{\mathfrak{p}}

\newcommand{\id}{\operatorname{id}}
\newcommand{\est} {\Lambda}

\newcommand{\Gl}{\operatorname{GL}}

\newcommand{\Ad}{\operatorname{Ad}}

\newcommand{\PP}{\mathbb{P}}

\renewcommand{\phi}             {\varphi}

\newcommand{\sieg}{\mathfrak{S}}

\newcommand{\Hg}{\operatorname{Hg}}
\newcommand{\siegv}{\sieg(V,\om)}
\newcommand{\siegl}{\sieg(V,\om)}

\newcommand{\Spv}{\operatorname{Sp}(V,\om)}
\newcommand{\spv}{\operatorname{\mathfrak{sp}}(V,\om)}
\newcommand{\mt}{\operatorname{MT}}
\newcommand{\smt}{\operatorname{Hg}}

\newcommand{\Sp}                {\operatorname {Sp}}

\newcommand{\liek}{\mathfrak{k}}

 \renewcommand{\Im}              {\operatorname{Im}}

\newcommand{\seconda}{\operatorname{II}}

\newcommand{\Mg}{\mathsf{M}_g}
\newcommand{\Mgs}{\mathsf{M}^*_g}

\newcommand{\A}{\mathsf{A}}
\newcommand{\Ag}{\mathsf{A}_g}

\newcommand{\M}{\mathsf{M}}

\newcommand{\ad}{{\operatorname{ad}}}

\newcommand{\molt}{\mathsf{m}}

\newcommand{\cc}{S}

\newcommand{\us}{\backslash}

\begin{document}

\author{Alessandro Ghigi}

\title{On some differential-geometric aspects of the Torelli map}

\begin{abstract}
  In this note we survey recent results on the extrinsic geometry of
  the Jacobian locus inside $\A_g$. We describe the second fundamental
  form of the Torelli map as a multiplication map, recall the relation
  between totally geodesic subvarieties and Hodge loci and survey
  various results related to totally geodesic subvarieties and the
  Jacobian locus.
\end{abstract}

\address{Universit\`{a} di Pavia} 
\email{alessandro.ghigi@unipv.it}

\thanks{The author was partially supported by MIUR PRIN 2015 ``Moduli
  spaces and Lie theory'' and by GNSAGA of INdAM.  }

\maketitle

\tableofcontents{}

\section{Introduction}

\begin{say}
  Let $\Mg$ denote the moduli space of smooth projective curves of
  genus $g$ and let $\Ag$ denote the moduli space of principally
  polarized abelian varieties of dimension $g$.  The \emph{Torelli
    map} $j : \Mg \ra \Ag$ associates to the point $[C] \in \Mg$ the
  moduli point of the jacobian of $C$ with the polarization induced
  from the cup product.  Both $\Mg$ and $\Ag$ have natural structures
  of quasi-projective varieties and $j$ is a regular map. By Torelli
  theorem it is injective.

  If one works over the complex numbers (as we do systematically),
  both $\Mg$ and $\Ag$ can be provided with the structure of complex
  analytic orbifold. (See \cite[XII,
  4]{arbarello-cornalba-griffiths-2} for the main definitions.)  This
  allows to work as if $\Mg$ and $\Ag$ were smooth. (Another
  possibility, that for our purposes is equivalent, is to fix level
  structures.)  In the following we will sometimes simplify the
  terminology by omitting the word "orbifold".

  The map $j$ is an orbifold map, i.e. it lifts to a holomorphic map
  of the uniformizers.  Oort and Steenbrink \cite {oort-steenbrink}
  proved that the restriction of $j$ to the set of non-hyperelliptic
  curves is an orbifold immersion.

  Next we recall that $\Ag$ has a natural metric. Indeed it is the
  quotient of the Siegel space $\sieg_g$, which is an irreducible
  Hermitian symmetric space of the non-compact type, by a properly
  discontinuous group of isometries. We call the induced metric on
  $\Ag$ the \emph{Siegel metric}.

  Summing up, if $\Mgs \subset \Mg$ denotes the complement of the
  hyperelliptic locus, then $j(\Mgs)$ is a complex analytic
  suborbifold of the Riemannian orbifold $\Ag$. It is natural to study
  the extrinsic geometry of $j(\Mgs)$ inside $\Ag$.  This study is
  still largely open and the goal of this note is to discuss some of
  the results obtained so far.  The rough idea behind these results is
  that $j(\Mgs)$ should be ``very curved'' inside $\Ag$. In other
  words the way in which $\Mg$ sits inside $\Ag$ should be
  ``complicated''.  This statement is extremely vague, but there are
  at least three ways to make it precise.
\end{say}

\begin{say}

  On the one hand the second fundamental form of the embedding
  $j: \Mgs \hookrightarrow \Ag$ should be highly nondegenerate,
  i.e. it should most of the time be non-zero.  This is far from
  understood. But there are some results on the second fundamental
  form.

  In \S \ref{seconda} we explain in some detail how the second
  fundamental form can be interpreted as a multiplication map. This is
  based on the fundamental work of Colombo, Pirola and Tortora
  \cite{cpt}.

\end{say}

\begin{say}

  On the other hand one might look at totally geodesic subvarieties of
  $\Ag$ and ask whether $j(\Mgs)$ contains some of the them. Here the
  expectation is that $j(\Mgs)$ should contain very few totally
  geodesic subvarieties.  The analogous statement for a surface in
  3-space is that the surface contains no line.  In the case of the
  Jacobian locus this expectation agrees with a rather famous
  conjecture, the Coleman-Oort conjecture, saying the $j(\Mg)$ should
  contain no Hodge locus of $\Ag$.

  In \S \ref{sec:siegel} we discuss these kind of problems.  First of
  all we prove that Hodge loci of $\Ag$ are totally geodesic
  subvarieties.  This is well-known, but it is hard to find an
  elementary exposition.

  Next we recall some non-existence results for totally geodesic
  subvarieties in $j(\Mg)$ based on the second fundamental form.  On
  the other hand we explain that in low genus there are some
  interesting examples of totally geodesic subvarieties generically
  contained in $\Ag$.
\end{say}

\begin{say}
  \label{terzo-rif}
  Using totally geodesic subvarieties in a different way we get to the
  third way of making precise the fact that $j(\Mg)$ is very curved
  inside $ \Ag$.  Consider a submanifold $M$ of a Riemannian manifold
  $A$.  One can look at the intersection of $M$ with totally geodesic
  submanifolds $Z \subset A$.  The fact that $M\cap Z$ has high
  codimension in $M$ for any $Z$ is our third way to express the
  complexity of the embedding $ M \hookrightarrow A$.  We explain this
  at the end of \S \ref{sec:siegel} and we describe a recent result
  saying that in the case of the embedding $j:\Mgs \subset \Ag$ the
  intersection of $j(\Mgs)$ with any totally geodesic subvariety has
  codimension at least 2.
\end{say}

\begin{say}
  This note is dedicated to the memory of my deeply esteemed teacher
  and friend Paolo de Bartolomeis.  While writing it I thought of
  Paolo so many times!  I was led to recall the glorious times when I
  was a student and I listened to Paolo's beautiful lectures.  I
  learned from him so many basic concepts!  Lie groups, Lie algebras,
  symmetric spaces, complex structures, symplectic forms, totally
  geodesic submanifolds and so on, just to mention the ones that are
  used continuously in this note.

  Paolo was really a friend. He had a wonderful sense of humour and I
  liked that a lot.  I was always happy when I was going to meet him
  at conferences, since talking with him was always very interesting
  and extremely pleasant.  Our last contact was by email. I had just
  watched for the first time a movie that Paolo liked a lot. I wrote
  him to tell that I also liked it a lot. His reply was great! Paolo
  was such a nice guy! I miss him a lot.

\end{say}

\medskip

{\bfseries \noindent{Acknowledgements}}.  The author wishes to thank
Professors L. Biliotti and G. P. Pirola for very interesting
discussions and Professor J. S. Milne for very interesting emails.

\section{The second fundamental form}
\label{seconda}

\begin{say}
  If $C$ is a non-hyperelliptic curve and $x=[C] \in \Mg$, then
  $dj_x : T_x\Mg \ra T_{j(x)}\Ag$ is injective and we have an exact
  sequence
  \begin{gather*}
    0 \lra T_x\Mg \stackrel{dj_x}{\lra} T_{j(x)} \Ag \stackrel{\pi}{
      \lra } N_x \lra 0,
  \end{gather*}
  where $N$ denotes the normal bundle to $j(\Mgs) \subset \Ag$.  The
  normal bundle in the Riemannian sense, that is as orthogonal
  complement, can be identified with the quotient bundle which is
  holomorphic.  Recall that $T_x \Mg \cong H^1(C,T_C)$ and
  $T^*_x \Mg \cong H^0(C,2K_C)$. If $A$ is a principally polarized
  abelian variety and $y=[A]\in \Ag$, then
  $T_{y} \Ag \cong S^2 H^0(A, T_A)$. If $y=j(x) $, i.e.
  $A= H^0(C,K_C) ^* / H_1(C,\Zeta)$, then $H^0(A, T_A) = H^0(C, K_C)$
  and $T_{j(x)} \Ag \cong S^2 H^0(C, K_C)^*$.  The transpose of the
  map $dj_x$ is the multiplication map
  \begin{gather*}
    \molt: T_{j(x)}^* \Ag = S^2H^0(C,K_C) \lra H^0(C,2K_C) \cong T_x^*
    \Mg.
  \end{gather*}
  Therefore $N_x^* = \Ann (\im j_x) $ is identified with $\ker \molt $
  and the dual of the above sequence is the following one:
  \begin{gather*}
    0 \lra I_2 (K_C) : =\ker \molt \lra S^2H^0(C,K_C)
    \stackrel{\molt}{\lra} H^0(C,2K_C) \lra 0.
  \end{gather*}
  (See \cite{cf1} for more details.).  Denote by
  \begin{gather*}
    \seconda_x : S^2 T_{x}\M_g=S^2H^1(C,T_C) \ra {N}_{x}
  \end{gather*}
  the second fundamental form of the Torelli embedding with respect to
  the Siegel metric on $\A_g$. We can identify $\seconda_x$ with a map
  \begin{gather}
    \label{ssd}
    \rho_{x} : N_x^*= I_2(K_C)\ra S^2 H^0(C, 2K_C).
  \end{gather}
  We will use the two symbols to distinguish the different
  interpretations, but they are the same object.
\end{say}

Our goal in this section is to interpret the map $\rho$ in \eqref{ssd}
as a multiplication map between spaces of sections on
$\cc:=C\times C$.

We start by explaining in some detail how to reinterpret domain and
target of $\rho$ as spaces of sections of appropriate bundles on
$\cc$.

Call $p$ and $q$ the two projections:
\begin{equation*}
  \begin{tikzcd}
    S \arrow{d}{p} \arrow{r}{q}  & C  \\
    C &
  \end{tikzcd}
  \qquad 
  \begin{aligned}
    p(x,y)=x,
    \\
    q(x,y)=y.
  \end{aligned}
\end{equation*}
Given line bundles $L\ra C$ and $M\ra C$, set
\begin{gather*}
  L\boxtimes M : = p^*L \otimes q^* M \lra S.
\end{gather*}
The map
\begin{gather*}
  H^0(C,L) \otimes H^0(C,M) \lra H^0(S, L \boxtimes M), \quad s\otimes
  t \mapsto p^*s \otimes q^* t,
\end{gather*}
is an isomorphism.  On $S$ we have the automorphism
\begin{gather*}
  \sigma : S \ra S , \quad \sigma (x,y) = (y,x).
\end{gather*}
For any line bundle $L\ra C$, $\sigma$ lifts to $L\boxtimes L$ as
follows:
\begin{gather*}
  \tilde{\sigma} : (L\boxtimes L)_{(x,y)} = L_x \otimes L_y \lra
  (L\boxtimes L)_{\sigma(x,y)} = L_y \otimes L_x
\end{gather*}
is simply the map
\begin{gather}
  \label{defsigmatilde}
  \tilde{\sigma}(u\otimes v) = v \otimes u, \quad u\in L_x , v\in L_y.
\end{gather}
Consider the special case $L=K_C$.  We have a canonical isomorphism
$K_S \cong K_C \boxtimes K_C$ given by the map
\begin{gather*}
  f: K_C\boxtimes K_C \lra K_S, \quad f(\alfa\otimes \beta ) =
  p^*\alfa \wedge q^* \beta , \quad \alfa \in T_x^*C, \quad \beta \in
  T_y^*C.
\end{gather*}
On $K_S$ we have two involutions lifting $\sigma$. One is simply the
pull-back: $ \sigma^* : K_S \lra K_S$. Since $p\sigma = q$, we have
\begin{gather*}
  \sigma^* ( p^*\alfa \wedge q^* \beta) = - p^*\beta \wedge q^*\alfa.
\end{gather*}
The other lift is $f \tilde{\sigma}f\meno : K_S \ra K_S$ and satisfies
\begin{gather*}
  f\tilde{\sigma}f\meno ( p^*\alfa \wedge q^* \beta) = f\tilde{\sigma}
  ( \alfa \otimes \beta) = f ( \beta \otimes \alfa) = p^*\beta \wedge
  q^*\alfa.
\end{gather*}
Hence
\begin{gather}
  \sigma^* = - f \tilde{\sigma}f\meno .\label{dueinvo}
\end{gather}
Consider now the isomorphism
\begin{gather*}
  f: H^0(C,K_C) \otimes H^0(C,K_C) \stackrel{\cong}{\lra} H^0(S,K_S),
\end{gather*}
induced by $f$ (and that we still denote by $f$).  It follows from
\eqref{dueinvo} that
\begin{gather*}
  H^0(S,K_S)^+: = \{\alfa \in H^0(S,K_S): \sigma^*\alfa = \alfa\} =   f (\est^2 H^0(C,K_C)) ,\\
  H^0(S,K_S)^-: = \{\alfa \in H^0(S,K_S): \sigma^*\alfa = -\alfa\} = f
  (S^2 H^0(C,K_C)) .
\end{gather*}
Let $\Delta \subset S$ denote the diagonal. It is a reduced divisor in
$S$.  We claim that
\begin{gather*}
  f(I_2(K_C) ) = \{\alfa \in H^0(S,K_S)^-: \alfa \restr{\Delta}=0\}.
\end{gather*}
To see this fix a coordinate system $z: U \subset C \ra \C$. From this
we get a chart $(z_1, z_2): U':=U\times U \subset S \ra \C^2$ by
setting $z_1 = z\circ p$ and $z_2 = z\circ q$. Further set
$x:= (z_1+ z_2 ) /2 $ and $y:= (z_1 - z_2)/2$. Then $(x,y)$ is another
coordinate system on $U'$, $\sigma (x,y) = (x, -y)$ and
$\Delta\cap U' = \{y=0\}$. Any $\alfa \in H^0(S,K_S)$ has a local
expression $\alfa = \phi(x,y)dx\wedge dy$ on $U'$ and
\begin{gather*}
  \sigma^*\alfa = - \phi(x,-y)dx \wedge dy.
\end{gather*}
Thus $\alfa \in H^0(S,K_S)^-$ iff $\phi(x,-y)=\phi(x,y)$ i.e. $\phi$
is an even function of $y$. In this case for any odd $m$ we have
\begin{gather*}
  \frac{\partial^m \phi }{\partial y^m} (x,0) \equiv 0.
\end{gather*}
It follows that if $\alfa\in H^0(S, K_S)^-$ vanishes along $\Delta$ it
vanishes there to second order. Hence
\begin{gather}
  \label{Iduecomesideve}
  I_2(K_C) = H^0(S,K_S(-2\Delta))^-.
\end{gather}

We can apply the same analysis to tensor powers of $K_C$.  Using the
same notation as above, we see that for any $n$ there is an
isomorphism
\begin{gather*}
  f_n: K^{n}_C\boxtimes K^n_C\stackrel{\cong}{ \lra} K^n_S, \\
  f_n(\alfa^n\otimes \beta^n ) = (p^*\alfa \wedge q^* \beta)^n , \quad
  \alfa \in T_x^*C, \beta \in T_y^*C.
\end{gather*}
Moreover for any $n$ there is a lifting $\sigma^*$ of $\sigma$ to
$K_S^n$.  By the same computation as above we get
\begin{gather*}
  \sigma^* = (-1)^n f_n \tilde{\sigma} f_n\meno.
\end{gather*}
Indeed for $\alfa\in T_xC^*$, $\beta \in T_yC^*$, we have
\begin{gather*}
  \tilde{\sigma} ( \alfa^n \otimes \beta^n) = \beta ^n \otimes
  \alfa^n,
\end{gather*}
as in \eqref{defsigmatilde}. Thus
\begin{gather*}
  f_n \tilde{\sigma} f_n\meno ((p^*\alfa \wedge q^*\beta )^n ) = f_n
  \tilde{\sigma} (\alfa^n \otimes\beta ^n ) =
  f_n   (\beta^n \otimes \alfa^n ) = \\
  =(p^*\beta \wedge q^* \alfa )^n = \sigma ^* ( (q^*\beta \wedge p^*
  \alfa )^n )= (-1)^n \sigma ^* ( ( p^* \alfa \wedge q^*\beta )^n ).
\end{gather*}
Thus
\begin{gather}
  \label{identificami}
  H^0(S,K^2_S)^+: = \{\alfa \in H^0(S,K_S): \sigma^*\alfa = \alfa\} =
  f_2 (S^2 H^0(C,K_C)) .
\end{gather}
Finally consider the line bundle $K_S(2\Delta) \lra S$. Since $\Delta$
is $\sigma$-invariant, there is a lift of $\sigma$ to $K_S(2\Delta) $,
that we still denote by $\sigma^*$.  We set
\begin{gather*}
  H^0(S, K_S(2\Delta))^\pm = \{ \alfa \in H^0(S, K_S(2\Delta)): \sigma
  ^* \alfa = \pm \alfa\}.
\end{gather*}

We are finally in the position to state the main theorem about the
second fundamental form.

\begin{teo}
  If $C$ is not hyperelliptic of genus at least 4, then there exists a
  section $\eta \in H^0(S, K_S(2\Delta) )^-$ such that, using
  \eqref{Iduecomesideve} and \eqref{identificami}, the second
  fundamental form \eqref{ssd} for $x=[C]$ gets identified with the
  multiplication map by $\eta$:
  \begin{gather*}
    \rho_x: H^0(S, K_S(-2\Delta))^- \lra H^0(S,2K_S)^+, \quad \rho_x(
    \alfa)= \eta \cd \alfa.
  \end{gather*}
\end{teo}

\begin{say}
  The proof of the Theorem is rather complicated and deep. The most
  important part is in \cite{cpt}. The reduction to a multiplication
  was achived in \cite{cf1,colombo-frediani-ghigi}.  Here we will only
  give some idea about the construction of $\eta$.

  Fix a curve $C$ of genus $g$ and a point $x \in C$.  The space
  $H^0(C, K_C (2x))$ is contained in the space of closed 1-forms on
  $C\setminus \{x\}$. The induced map
  $H^0(C,K_C(2x)) \ra H^1(C \setminus \{x\},\C)$ is injective as soon
  as $g>0$. By Mayer-Vietoris
  $H^1(C,\C) \cong H^1(C\setminus \{x\}, \C)$. Hence we get an
  injection $ H^0(C,K_C(2x)) \hookrightarrow H^1(C,\C)$.  We identify
  $ H^0(C,K_C(2x))$ with its image inside $H^1(C,\C)$. The space
  $H^{1,0}(C)$ is contained in $H^0(C,K_C(2x))$. Since
  $h^0(C,K_C(2x)) = g+1$, the intersection
  $ H^0(C, K_C(2x)) \cap H^{0,1}(C)$ is a 1-dimensional.  If
  $u \in T_xC$ we choose a local coordinate at $x$ such that
  $u=\partial/\partial z$.  Then there is a unique
  $\phi_u \in H^0(C, K_C(2x))\cap H^{0,1}$ such that locally
  $ \phi_u = f(z) dz $ with $f(z)= 1/z^{2} + h(z) $ and $h$ is
  holomorphic.  Note that $\phi_u$ only depends on $u$, not on the
  coordinate $z$. (If $u=0$, set $\phi_u = 0$.)  Moreover the map
  $u \mapsto \phi_u$ is linear.  Consider the vector bundle
  $V:= p_* (q^* K_C (2\Delta))$ over $C$. We have
  $V_x = H^0(C, K_C(2x))$.  The maps $T_xC \ra V_x$,
  $u \mapsto \phi_u$ give a holomorphic section of
  $K_C \otimes V = p_* (p^*K_C \otimes q^*K_C(2\Delta)) = p_*(K_S
  (2\Delta))$,
  \cite[Prop. 3.4]{colombo-frediani-ghigi}.  But global sections of
  $ p_*(K_S (2\Delta))$ are the same as global sections of
  $K_S(2\Delta)$. Thus we get a global section of $K_S(2\Delta)$ and
  this is the form $\eta$ in the theorem!

  The definition of $\eta$ clearly involves Hodge theory. It is not
  clear how to control the behaviour of the form $\eta$ at points in
  $S - \Delta$. On the other hand the behaviour of $\eta$ along the
  diagonal is closely related to the second Wahl map, which is an
  algebraic object.

\end{say}

\section{Totally geodesic subvarieties and Hodge loci in $\Ag$}
\label{sec:siegel}

In this \S {} we are interested in totally geodesic submanifolds of
Siegel space.  The starting point is the following characterization of
totally geodesic submanifolds in symmetric spaces. (See \cite
[p. 19]{eschenburg-symspace} for a proof.)

\begin{teo}
  \label{totg} Let $X$ be a symmetric space and let $X' \subset X$ be
  a closed connected submanifold. Set $G:=\operatorname{Isom}(X)^0$.
  Fix a point $o\in X'$, set $K:=G_o$ and let
  $\lieg = \liek \oplus \liep$ be the Cartan decomposition.  The
  following conditions are equivalent.
  \begin{enumerate}
  \item $X'$ is totally geodesic.
  \item For any $x \in X'$ we have $s_x(X') = X'$.
  \item $X'=\exp_o \liem$, where $\liem \subset \liep$ is a Lie triple
    system i.e. $[[\liem, \liem],\liem ]\subset \liem$.
  \end{enumerate}
\end{teo}

\begin{say}\label{say-siegel}
  We recall the definition and the main properties of Siegel space.
  Fix a real symplectic vector space $(V,\om)$ of dimension $2g$.  If
  $J$ is a complex structure on $V$, we let $V_J$ denote the complex
  vector space obtained using $V$ as underlying real vector space and
  letting multiplication by $i$ act as $J$.  We have $J^*\om = \om$ if
  and only if the bilinear form $g_j:= \om (\cd, J\cd )$ is symmetric.
  Set
  \begin{gather*}
    \sieg(V,\om) := \{ J\in \End V: J^2 = -\id_V, J^*\om=\om, \\ g_J
    \text{ is positive definite}\}.
  \end{gather*}
  If $J\in \siegv$, then $H_J (x,y):= g_J(x,y) - i \om (x,y)$ is a
  Hermitian product on $V_J$.  The group $\Spv$ acts on $\siegv$ by
  conjugation: $a\cd J :=a J a\meno$.  We claim that this action is
  transitive. Indeed let $J, J'\in \siegv$.  Fix an $H_J$-unitary
  basis $\{u_i\} $ of $V_J$. Similarly let $\{u'_i\}$ be an
  $H_{J'}$-unitary basis of $V_{J'}$. Setting $b (u_i) : = u'_i$ we
  get a complex isometry $ b : V_J \ra V_{J'}$, so $b^*H_{J'} = H_J$.
  Let $a \in \End V$ denote the underlying real isomorphism. It
  satisfyes $a^*\om = - a^*(\Im H_{J'}) = -\Im H_J = \om$ and
  $a J = J' a$. Thus $a\in \Spv$ and $a\cd J = J'$. This proves the
  claim.  It is easy to check that if $a\in \Spv$, then $a \cd J = J$
  if and only if $a^*H_J = H_J$. It follows that the stabilizer of $J$
  in $\Spv$ is the unitary group $ \operatorname{U}(V_J,H_J)$.  For
  $L \in \End V$, let $L^{T_J} $ be the transposed operator with
  respect to the scalar product $g_J$. We claim that for $a\in \Spv$
  \begin{gather}
    \label{tetaj}
    a^{T_J} = -J a\meno J = J a\meno J\meno.
  \end{gather}
  This follows easily using $\om = g_J (J\cd, \cd )$ and
  $a^*\om =\om $.  Thus
  \begin{gather*}
    \theta_J(a) : = (a\meno)^{T_J} = J a J\meno,
  \end{gather*}
  is a Cartan involution on $\Sp(V,\om)$, At the Lie algebra level
  $\theta_J = \Ad J$.  The $1$-eigenspace of $\theta_J$ on $\spv$ is
  $\mathfrak{u}(V_J,H_J)$. This show that $\siegv$ is a symmetric
  space, see \cite[p. 209]{helgason}.

\end{say}

\begin{say}
  Up to now we only used the symplectic vector space $(V,\om)$.  Now
  we add an extra structure, namely we fix a lattice $\La \subset V$
  and we assume that $\om ( \La \times \La ) \subset \Zeta$.  In
  appropriate bases of $\La$ the matrix of $\om$ has the form
  \begin{gather*}
    \begin{pmatrix}
      0 & T \\ -T & 0
    \end{pmatrix} \quad \text{ with } T=\operatorname {diag} (d_1,
    \lds, d_g)
  \end{gather*}
  and $0 < d_1 | d_2 | \cds |d_g$, see
  e.g. \cite[p. 391]{freitag-funk-2}. The vector $(d_1, \lds, d_g)$ is
  called the \emph{type} of the form $\om$. Our final assumption is
  that $\om$ has type $(1, \lds, 1)$.  The group
  $\Ga:= \Sp (\La, \om) $ is a discrete subgroup of $\Spv$ and acts
  properly discontinuously on $\siegl$. Set
  \begin{gather*}
    \Ag:= \Ga \us \siegv.
  \end{gather*}
  By a theorem of H. Cartan this quotient is a normal complex analytic
  space. It has also a natural structure of complex analytic
  orbifold. It can be shown that it is a quasiprojective variety.
  Since $\Ga $ is a group of isometries, the symmetric Riemannian
  structure on $\siegv$ induces a locally symmetric orbifold metric on
  $\Ag$, which we call \emph{Siegel} metric.
\end{say}

\begin{say} \label{saymt} Once the lattice $\La \subset V$ has been
  fixed there is a natural (and tautological) $\Zeta$-variation of
  Hodge structure on $\siegv$: take the constant lattice $\La$ and for
  $J\in \siegv$ consider the Hodge structure $(\La, V_J^{1,0})$.
  Since this Hodge structure only depends on $J$, we denote it simply
  by $J$.  This variation of Hodge structure descends to an (orbifold)
  variation over $\Ag$.  We are interested in the Hodge loci of this
  variation on $\Ag$.  For the main definitions and facts regarding
  Hodge loci and Mumford-Tate groups we refer to
  \cite{moonen-oort,schnell-trento,voisin-hodge-loci}.  Here we recall
  what we need only in the case of weight 1.

  Given $J\in \siegv$ we define a representation
  $\rho_J : \C^* \ra \Gl(V)$, setting $\rho_J(z) v = z \cd v$ for
  $v \in H^{1,0}$ and $\rho_J(z) v = \bar{z} \cd v$ for
  $v \in H^{0,1}$.  The \emph{Mumford-Tate group} of $J$, denoted
  $\mt(J)$, is the smallest algebraic subgroup of $\Gl(V)$ defined
  over $\QQ$, whose real points contain $\im \rho_J$.  The main
  property of the Mumford-Tate group is the following: given
  multi-indices $d, e \in \mathbb{N}^m$ set
  \begin{gather*}
    T^{d,e} (\La_\QQ):= \oplus_{j=1}^m \La_\QQ^{\otimes d_j} \otimes
    (\La_\QQ^*)^{\otimes e_j}.
  \end{gather*}
  This space is a pure Hodge structure.  A vector
  $v \in T^{d,e}(\La_\QQ)$ is invariant by $\mt(H) $ if and only if it
  is a Hodge class of type $(0,0) $, i.e a rational vector of type
  (0,0).  Moreover the Mumford-Tate group is characterized by this
  property in the following sense: if $G \subset \Gl(H)$ is the
  subgroup containing the elements that fix the Hodge classes of
  $T^{d,e}$ for any $d$ and $e$, then $G=\mt(H)$, see e.g.
  \cite{schnell-trento}.  The \emph{Hodge group} or \emph{special
    Mumford-Tate group}, denoted $\smt(J)$ is the smallest algebraic
  subgroup of $\Gl(V)$ defined over $\QQ$, whose real points contain
  $ \rho_J(S^1)$.  If $D$ denotes the subgroup of diagonal matrices in
  $\Gl(V)$, then $\mt(J) = D \cd \smt(J)$.

\end{say}

\begin{lemma}
  \label{Hgg}
  For any $J\in \siegv$, the following properties hold.
  \begin{enumerate}
  \item $J \in \Hg(J) \cap \Lie \Hg(J)$.
  \item $\smt( J)$ is invariant by $\theta_J$.
  \item The stabilizer of $J $ in $\smt(J)$ coincides with the
    centralizer of $J$ in $\smt(J)$ and it is a maximal compact
    subgroup of $\smt(J)$.
  \item The orbit $\smt(J)\cd J \subset \siegv$ is a complex totally
    geodesic submanifold of $\siegv$.  With the induced metric it is a
    Hermitian symmetric space of the non-compact type.
  \end{enumerate}
\end{lemma}
\begin{proof}
  $J \in \rho_J(S^1) \subset \smt(J)$. Moreover
  $\rho_J(e^{it}) = \cos t \cd I + \sin t \cd J$. Thus
  $J = d\rho (0) (i) \in \Lie \smt(J)$. This proves (1).  To prove (2)
  use \eqref{tetaj} and the fact that $J \in \Hg(J)$: for
  $a \in \smt(J)$, $\theta_J(a) = Ja\meno J\meno \in \smt(J)$.  The
  restriction of $\theta_J$ to $\Hg(J)$ is a Cartan involution on
  $\smt(J)$.  If $a\in \Sp(V, \om)$, then $\theta_J(a) = a $ iff
  $aJ=Ja$ iff $a\cd J = J$.  Thus the stabilizer of $J$ in $\smt(J)$
  is the fixed set of $\theta_J$ in $\smt(J)$, which is a maximal
  compact subgroup. This proves (3).  Set $\lieg :=\Lie \Hg(J)$ and
  let $\lieg = \liek \oplus \liep$ be the Cartan decomposition
  corresponding to $\theta_J$. Then $[\liep, \liep] \subset \liek$ and
  $[\liek, \liep] \subset \liep$. Thus $\liep$ is a Lie triple system
  and $\Hg(J) \cd J = \exp_J \liep$ is a totally geodesic submanifold
  by Theorem \ref {totg}.  It is complex since $\ad J$ preserves
  $\liep$.
\end{proof}

\begin{say}
  \label{locisay}
  The {Hodge loci} of the natural variation of Hodge structure on
  $\Ag$ are defined as follows.  Given $d, e$ and
  $t\in T^{d,e}(\La_\QQ)$, set
  \begin{gather*}
    Y(t):=\{ J \in \siegv: t \in T^{d,e}(V_J)^{0,0} \}.
  \end{gather*}
  $Y(t)$ is an analytic subset of $\siegv$, see
  \cite[p. 404]{voisin-libro}. If $t_1, \lds, t_r$ are rational
  vectors in various tensor constructions, set
  $Y(t_1, \lds, t_r) := Y(t_1) \cap \cds \cap Y(t_r)$.  We call
  $Y(t_1, \lds, t_r)$ \emph{proper} if
  $Y(t_1, \lds, t_r) \neq \siegv$.  Let \begin{gather*} \pi: \siegv
    \lra \Ag,
  \end{gather*}
  denote the canonical projection.  A \emph{Hodge locus} of $\Ag$ is
  an irreducible component of a proper
  $\pi (Y(t_1, \lds , t_r)) \subsetneq \Ag$.  It is easy to check that
  Hodge loci are exactly the subsets of the form $\pi(Z)$, where $Z$
  is an irreducible component of some proper $ Y(t_1, \lds, t_r)$.
  The irreducible components of proper subsets $ Y(t_1, \lds, t_r) $
  form a countable family $\{Z_i\}_{i\in \mathbb{N}}$ of proper
  subsets of $ \siegv$.  Set further
  \begin{gather*}
    Z_i^0 := Z_i \setminus \bigcup_{j: Z_i \not \subset Z_j} Z_j.
  \end{gather*}
\end{say}

\begin{teo}
  \label{teoZi}
  For $J \in Z_i^0$ we have $Z_i = \smt(J)\cd J$.  In particolar $Z_i$
  is a totally geodesic submanifold of $\siegv$.
\end{teo}
\begin{proof}
  We claim that the set of Hodge classes is constant on
  $Z_i^0$. Indeed let $x, y \in Z_i^0 $ be distinct points. Assume
  that $Z_i$ is a component of $Y(t_1, \lds, t_r)$ and assume by
  contradiction that there is $t_{r+1}$ that is a Hodge class at $x$,
  but not at $y$. Let $Z' $ be the irreducible component of
  $Y(t_1, \lds, t_r,t_{r+1})$ containing $x$.  Then $y \not \in Z'$,
  so $Z'$ is among the $Z_j$'s with $Z_i \not \subset Z_j$. But then
  $x \not \in Z_i^0$.  This proves the claim.  By the property
  mentioned in \ref{saymt} we conclude that the Mumford-Tate and Hodge
  groups are constant on $Z_i^0$.  Fix $J_0\in Z_i^0$ and set
  $G:=\smt(J_0)$. We just proved that
  \begin{gather*}
    Z_i^0 \subset Y:= \{ J \in \siegv : \smt(J) = G\}.
  \end{gather*}
  Set $K:= G_{J_0}$ and $A:=Z(K)$. We claim that $A \cap \siegv$ is a
  finite set.  Indeed $A$ is abelian and acts unitarily and faithfully
  on $(V_{J_0}, H_{J_0})$. Hence $V_{J_0}$ splits in one-dimensional
  subrepresentations $V_i$ on which $A$ acts by a character
  $\chi_i$. But if $J \in A\cap \siegv$, then $J^2 = -I$, so
  $\chi_i (J)= \pm i$ for any $i$.  Since the representation is
  faithful this shows that $A \cap \siegv$ is finite.

  Since $Z_i$ is irreducible, $Z_i^0$ is connected.  Let $Y'$ be the
  connected component of $Y$ that contains $Z_i^0$.  We claim that
  \begin{gather}
    \label{zweiter-claim}
    Y' \subset G\cd J_0.
  \end{gather}
  Indeed let $J $ be a point in $Y'$. Then $\smt(J)=G$. So $G$ is
  $\theta_J$-invariant and the stabilizer $G_J$ is a maximal compact
  subgroup of $G$. Since $G$ is connected there is $a \in G$ such that
  $a\meno G_J a = K$. Moreover $J$ belongs to the center of $G_J$,
  hence $a\meno J a $ belongs to the center of $K$. This shows that
  $ a\meno J a \in A\cap \siegv$, which is a finite set. Since $Y'$ is
  connected we have necessarily $a\meno J a = J_0$, i.e.
  $J = a \cd J_0$. This proves \eqref{zweiter-claim}.

  Finally we claim that $G\cd J_0 \subset Z_i$.  Assume that
  $J= a\cd J_0$ with $a\in G$.  Then $J = \Ad(a) (J_0) \in \lieg$, so
  $\rho_J(e^{it}) \in G$ for any $t$.  Hence $\smt(J) \subset G$.  By
  assumption $Z_i$ is an irreducible component of some proper subset
  $Y(t_1, \lds, t_r)$.  Then by \ref{saymt} we have
  \begin{gather*}
    \mt(J_0) = \{a \in \Gl(\La_\QQ): a \cd t_j = t_j \text{ for }
    j=1,\lds, r\}.
  \end{gather*}
  Since $\smt(J) \subset G$, $\mt(J) \subset \mt(J_0)$, so
  $J \in Y(t_1, \lds, t_r)$.  We have proved that
  $G\cd J_0 \subset Y(t_1, \lds, t_r)$.  Since $G\cd J_0$ is
  irreducible, we have in fact $G\cd J_0 \subset Z_i$.

  We have proved the inclusions
  \begin{gather*}
    Z_i^0 \subset Y' \subset G\cd J_0 \subset Z_i.
  \end{gather*}
  Since $G\cd J_0$ is a closed subset of $\siegv$ and
  $\overline{Z_i^0 } = Z_i$, we conclude that $G\cd J_0 =Z_i$ as
  desired.  The last statement follows from Lemma \ref{Hgg} (4).
\end{proof}

  \begin{defin}
    A \emph{totally geodesic subvariety} of $\A_g$ is a closed
    algebraic subvariety $Z \subset \A_g$, such that $Z = \pi (X)$ for
    some totally geodesic submanifold $X \subset \siegv$.
  \end{defin}

\begin{cor}
  Hodge loci of $\Ag$ are totally geodesic subvarieties.
\end{cor}
\begin{proof}
  Let $W \subset \Ag$ be a Hodge locus.  By the theorem of
  Cattani-Deligne-Kaplan \cite{aroldo} $W$ is a closed algebraic
  subset of $\Ag$.  As noticed in \ref {locisay} $W=\pi (Z_i)$ for
  some $Z_i \subset \siegv$ for some irreducible component $Z_i$.  By
  Theorem \ref{teoZi} $Z_i$ is a totally geodesic submanifold of
  $\siegv$.
\end{proof}

Hodge loci of $\Ag$ are also called \emph{special subvarieties} or
\emph{Shimura subvarieties}.  A CM point of $\Ag$ is by definition a
moduli point $[A]$, where $A$ is an abelian variety with complex
multiplication, see e.g. \cite{mumford-Shimura}.  This condition is of
arithmetic nature.  Shimura varieties always contain CM points. This
condition in fact characterises them among totally geodesic
subvariety, see \cite{mumford-Shimura,moonen-linearity-1}.  For this
reason Shimura varieties play a prominent role in arithmetic algebraic
geometry. We say that a subvariety $Z \subset \Ag$ is
\emph{generically contained} in $j(\Mg)$, if
$Z \subset \overline{j(\Mg)}$ and $Z \cap j(\Mg) \neq \vacuo$.  The
following conjecture is rather important in arithmetic algebraic
geometry.
\begin{conj}
  [Coleman-Oort] For large $g$ there is no Shimura variety
  $Z \subset \Ag$ generically contained in $j(\Mg)$.
\end{conj}

\begin{say}

  Using the results on the second fundamental form one can get some
  constraints on the existence of totally geodesic subvarieties of
  $\A_g$ contained in $\M_g$. Since these methods are of local nature,
  the results apply to analytic germs of such subvarieties.
\end{say}

\begin{teo}
  Assume that $C$ is a $k$-gonal curve of genus $g$ with $g\geq 4$ and
  $k\geq 3$. Let $Y$ be a germ of a totally geodesic subvariety of
  $A_g$ which is contained in the Jacobian locus and passes through
  $j([C])$. Then $\dim(Y) \leq 2g+k - 4$.
\end{teo}
This immediately yields a bound which only depends on $g$.

\begin{teo}
  If $g\geq 4$ and $Y$ is a germ of a totally geodesic subvariety of
  $A_g$ contained in the Jacobian locus, then
  $\dim Y \leq \frac{5}{2}(g-1)$.
\end{teo}

Thus the existence of totally geodesic subvarieties (and in particular
of Shimura varieties) of very large dimension is excluded. This agrees
with the Coleman-Oort conjecture. One in fact expects a much better
bound to hold than the one in the previous theorem. But up to now that
is the best known.

\begin{say}
  The Coleman-Oort conjecture precludes the existence of Shimura
  varieties generically contained in the Jacobian locus for large
  genus. But for low genus, namely for $g\leq 7$, one can construct
  examples of totally geodesic subvarieties contained in $\M_g$.  Most
  of these examples are constructed using families of Galois covers of
  $\PP^1$ and are in fact Shimura varieties.  The first examples
  obtained in this way were cyclic covers, see
  e.g. \cite{dejong-noot,moonen-special,rohde}.  A complete list of
  the Shimura varieties that can be obtained using cyclic covers has
  been given by Moonen \cite{moonen-special} using deep results in
  positive characteristic.  It would be interesting to have a simple
  differential-geometric proof of this result.  Some results in that
  direction are contained in \cite{colombo-frediani-ghigi}, but at the
  moment one there is no proof of Moonen's result using differential
  geometry.  Other examples of Shimura varieties generically contained
  in the Jacobian locus were later constructed using non-cyclic Galois
  covers of $\PP^1$, see \cite{moonen-oort} and
  \cite{frediani-ghigi-penegini}.  Finally some examples were gotten
  using Galois cover of elliptic curves, see
  \cite{frediani-penegini-porru}. This paper studies in particular a
  3-dimensional Shimura variety generically contained in $\M_4$.  This
  variety was first constructed by Pirola \cite{pirola-Xiao} to
  disprove a conjecture of Xiao.  The same variety has been studied in
  \cite{grushevsky2015explicit}, where it is shown that it is fibered
  in totally geodesic curves.  Since CM points are countable, only
  countably many fibres are Shimura varieties.  Therefore most of
  these fibres are totally geodesic curves that are not Shimura.
  Obvioulsy they are generically contained in the Jacobian locus.

\end{say}

  \begin{say}
    Other works studying totally geodesic subvarieties in the Jacobian
    locus include \cite{toledo,hain,dejong-zhang,liu-yau-ecc,
      chen-lu-zuo-Compositio,lu-zuo-Mumford-prep, lu-zuo-HE,
      grushevsky-moeller-genus-3,mohajer-zuo,ppp}.  The papers
    \cite{cfgp} and \cite{cf2} consider the corrisponding problem for
    the Prym locus instead of the Jacobian locus.
  \end{say}

  \begin{say}
    As mentioned in the introduction, the idea behind all the results
    we have recalled is that the way in which $\Mg$ sits inside $\Ag$
    should be rather ``complicated''.  Here we wish to make this
    statement more precise in the third way sketched in
    \ref{terzo-rif}.  Consider a Riemannian manifold $A$ and a
    submanifold $M \subset A$.  If the manifold $A$ has a lot of
    totally geodesic submanifolds (and this is the case for symmetric
    spaces) the study of the intersections $M\cap Z$, where $Z$ is a
    totally geodesic submanifold of $A$, gives information on $M$.
    The first question one asks in this setting is whether there is a
    totally geodesic $Z$ such that $M\cap Z = M$, i.e. $M\subset Z$.
    If this does not happen one says that $M$ is \emph{full}.  In
    Euclidean space totally geodesic submanifolds are affine
    subspaces, so being full means that $M \subset \R^n$ is not
    contained in a hyperplane.  If $M$ is full, then for any totally
    geodesic $Z$ the intersection $M\cap Z$ is a possibly singular
    proper submanifold of $M$.  One might consider the totally
    geodesic submanifolds of $A$ as analogues of affine linear
    subspaces.  If the codimension of $M\cap Z$ is always $\geq k$.
    not only $M$, but also all its submanifolds of codimension $> k$,
    are full. This is a measure of the complexity of $M$ in $A$, at
    least when $A$ has a lot of totally geodesic submanifolds.

    The following result was obtained recently in
    \cite{frediani-ghigi-pirola}.

  \end{say}

  \begin{teo}
    \label{notg}
    If $g\geq 3$ and $Y \subset \M_g$ is a divisor, then $j(Y)$ is
    full in $\A_g$.
  \end{teo}
  This shows that in the case of the embedding
  $j: \Mgs \hookrightarrow \Ag$ $k\geq 2$, i.e. for every totally
  geodesic subvariety $\Mg\cap Z$ has codimension at least 2 in $\Mg$.
  The proof is based on induction on the genus. The case $g=3$ follows
  from the fact that $\sieg_3$ contains no totally geodesic
  divisor. The inductive step depends on algebro-geometric techniques
  developed in \cite{marcucci-naranjo-pirola} and on simple Lie
  theoretic arguments.  It is an interesting problem to get better
  estimates for $k$.

  \def\cprime{$'$}

\end{document}